\documentclass[11pt,reqno]{amsart}

\usepackage{latexsym,bm, amsfonts,amsthm,amsmath,mathrsfs,CJK}

 \newtheorem{lem}{Lemma}[section]  \newtheorem{thm}{Theorem}[section]
   
\numberwithin{equation}{section}

\numberwithin{equation}{section}

 \DeclareMathAlphabet{\mathsfsl}{OT1}{cmss}{m}{sl} \DeclareMathAlphabet{\mathpzc}{OT1}{pzc}{m}{it}

    \newcommand{\rr}{\mathbb{R}}
\newcommand{\pp}{\mathbb{P}}    

    \def\FF{\mathcal F}

 \def\d"{^{\prime\prime}} \def\d'{^{\prime}}

\begin{document}

\title[]{Central Limit Theorem and Moderate deviation for nonhomogenenous Markov chains}
\thanks{The research  is supported by  Scientific Program of Department of Education of Jiang Xi Province of China  Grant GJJ150894} 
\date{} \maketitle

\begin{center}

XU Mingzhou~\footnote{Email: mingzhouxu@whu.edu.cn}\quad DING Yunzheng~
  \quad ZHOU Yongzheng~\footnote{Email:zhyzh\_ty@163.com} \\

 School of Information Engineering, Jingdezhen Ceramic Institute\\
  Jingdezhen 333403, China

\end{center}
\begin{abstract}
Our purpose is to prove central limit theorem for countable nonhomogeneous Markov chain  under the condition of uniform convergence of transition probability matrices for countable nonhomogeneous Markov chain in Ces\`aro sense. Furthermore, we obtain a corresponding moderate deviation theorem  for countable nonhomogeneous Markov chain by G\"artner-Ellis theorem and exponential equivalent method.
\end{abstract}

\section{Introduction and main results}
Huang et al. {\cite{H2013}} proved central limit theorem for nonhomogeneous Markov chain with finite state space. Gao \cite{Gao1992} obtained moderate deviation principles for homogeneous Markov chain. De Acosta \cite{deAcosta1997} studied moderate deviations lower bounds for homogeneous Markov chain. De Acosta and Chen \cite{deAcosta1998} established moderate deviations upper bounds for homogeneous Markov chain. It is natural and important to study central limit theorem and moderate deviation for countable nonhomogeneous Markov chain. We wish to investigate a central limit theorem and moderate deviation for countable nonhomogeneous Markov chain under the condition of uniform convergence of transition probability matrices for countable nonhomogeneous Markov chain in Ces\`aro sense.

Suppose that $\{X_n,n\ge 0\}$ is a nonhomogeneous Markov chain taking values in  $S=\{1,2,\ldots\}$ with initial probability
\begin{equation}\label{1}
\mu^{(0)}=(\mu(1),\mu(2),\ldots)
\end{equation}
and the transition matrices
\begin{equation}\label{2}
P_n=(p_n(i,j)),\mbox{  }i,j\in S, n\ge 1,
\end{equation}
where $p_n(i,j)=\pp(X_n=j|X_{n-1}=i)$. Write
$$
P^{(m,n)}=P_{m+1}P_{m+2}\cdots P_{n}, p^{(m,n)}(i,j)=\pp(X_n=j|X_m=i),
$$
$$
\mu^{(k)}=\mu^{(0)}P_1P_2\cdots P_k, \mu^{(k)}(j)=\pp(X_k=j).
$$When the Markov chain is homogeneous,  $P$,  $P^k$ denote $P_n$, $P^{(m,m+k)}$ respectively.

If $P$ is a stochastic matrix, then we write
$$
\delta(P)=\sup_{i,k}\sum_{j=1}^{\infty}[p(i,j)-p(k,j)]^{+},
$$
where $[a]^{+}=\max\{0,a\}$.

Let $A=(a_{ij})$ be a matrix defined as $S\times S$. Write
$$
\|A\|=\sup_{i\in S}\sum_{j\in S}|a_{ij}|.
$$

If $h=(h_1,h_2,\ldots)$ , then we write $\|h\|=\sum_{j\in S}|h_j|$. If $g=(g_1,g_2,\ldots)'$ , then we write $\|g\|=\sup_{i\in S}|g_i|$. The properties below hold ( Yang \cite{Yang2002,Yang2009}):

(a)  $\|AB\|\le \|A\|\|B\|$ for all matrices $A$ and $B$;

(b) $\|P\|=1$ for all stochastic matrix $P$. 

Suppose that $R$ is a 'constant' stochastic matrix each row of which is the same. Then $\{P_n,n\ge 1\}$ is said to be strongly ergodic (with a constant stochastic matrix $R$) if for all $m\ge 0$
$$
\lim_{n\rightarrow\infty}\|P^{(m,m+n)}-R\|=0.
$$
The sequence $\{P_n,n\ge 1\}$ is said to converge in the Ces\`{a}ro sense (to a constant stochastic matrix $R$) if for every $m\ge 0$
$$
\lim_{n\rightarrow\infty}\left\|\sum_{t=1}^{n}P^{(m,m+t)}/n-R\right\|=0.
$$
The sequence $\{P_n,n\ge 1\}$ is said to uniformly converge in the Ces\`{a}ro sense (to a constant stochastic matrix $R$) if
\begin{equation}\label{3}
\lim_{n\rightarrow\infty}\sup_{m\ge 0}\left\|\sum_{t=1}^{n}P^{(m,m+t)}/n-R\right\|=0.
\end{equation}

 $S$ is divided into $d$ disjoint subspaces $C_0$, $C_1$, $\ldots$, $C_{d-1}$, by an irreducible stochastic matrix $P$, of period $d$ ($d\ge 1$) (see Theorem 3.3 of Hu \cite{Hu1983}), and $P^d$ gives $d$ stochastic matrices $\{T_l,0\le l\le d-1\}$, where $T_l$ is defined on $C_l$. As in Bowerman, et al. \cite{Bowerman1977} and Yang \cite{Yang2002}, we shall discuss such an irreducible stochastic matrix $P$, of period $d$  that $T_l$ is strongly ergodic for $l=0,1,\ldots,d-1$. This matrix will be called periodic strongly ergodic.

\begin{thm}\label{thm1} Suppose $\{X_n,n\ge0\}$ is a countable nonhomogeneous Markov chain taking values in $S=\{1,2,\ldots\}$ with initial distribution of (\ref{1}) and transition matrices of (\ref{2}). Assume that $f$ is a real function satisfying $|f(x)|\le M$ for all $x\in \rr$. Suppose that $P$ is a periodic strongly ergodic stochastic matrix. Assume that $R$ is a constant stochastic matrix each row of which is the left eigenvector $\pi=(\pi(1),\pi(2),\ldots)$ of $P$ satisfying $\pi P=\pi$ and $\sum_{i}\pi(i)=1$. Assume that
\begin{equation}\label{4}
\lim_{n\rightarrow\infty}\sup_{m\ge 0}\frac1n\sum_{k=1}^{n}\|P_{k+m}-P\|=0,
\end{equation}
and
\begin{equation}\label{5}
\theta=\sum_{i\in S}\pi(i)[f^2(i)-(\sum_{j\in S}f(j)q(i,j))^2]>0.
\end{equation}
Moreover, if the sequence of $\delta$-coefficient satisfies
\begin{equation}\label{6}
\lim_{n\rightarrow\infty}\frac{\sum_{k=1}^{n}\delta(P_k)}{\sqrt{n}}=0,
\end{equation}
then we have
\begin{equation}\label{7}
\frac{S_n-E(S_n)}{\sqrt{n\theta}}\stackrel{D}{\Rightarrow}N(0,1)
\end{equation}
where $S_n=\sum_{k=1}^{n}f(X_k)$, $\stackrel{D}{\Rightarrow}$ stands for the convergence in distribution.
\end{thm}

\begin{thm}\label{thm2} Under the hypotheses of Theorem \ref{thm1}, if moreover
\begin{equation}\label{8}
\lim_{n\rightarrow\infty}\frac{a(n)}{\sqrt{n}}=\infty,\lim_{n\rightarrow\infty}\frac{a(n)}{n}=0.
\end{equation}
then for each open set $G\subset \rr^1$,
$$
\lim_{n\rightarrow\infty}\frac{n}{a^2(n)}\log\pp \left\{\frac{S_n-E(S_n)}{\sqrt{a(n)}}\in G\right\}\ge -\inf_{x\in G}I(x),
$$
and for each closed set $F\subset \rr^1$,
$$
\lim_{n\rightarrow\infty}\frac{n}{a^2(n)}\log\pp \left\{\frac{S_n-E(S_n)}{\sqrt{a(n)}}\in F\right\}\le -\inf_{x\in F}I(x),
$$
where $I(x):=\frac{x^2}{2\theta}$.
\end{thm}
In sections 2 and 3, we prove Theorems \ref{thm1} and \ref{thm2}. The ideas of proofs of Theorem \ref{thm1} come from Huang et al. \cite{H2013} and Yang \cite{Yang2002}.
\section{Proof of Theorem \ref{thm1}}
Let
\begin{equation}\label{2.1}
D_n=f(X_n)-E[f(X_n)|X_{n-1}], n\ge 1, \mbox{  } D_0=0
\end{equation}
\begin{equation}\label{2.2}
W_n=\sum_{k=1}^{n}D_k.
\end{equation}
Write $\FF_n=\sigma(X_k,0\le k\le n)$. Then $\{W_n,\FF_n,n\ge1\}$ is a martingale, so that $\{D_n,\FF_n,n\ge 0\}$ is the related martingale difference. For $n=1,2,\ldots$, set
$$
V(W_n):=\sum_{k=1}^{n}E[D_k^2|\FF_{k-1}],
$$
and
$$
v(W_n):=E[V(W_n)].
$$
It is clear that
$$
v(W_n)=E[W_n^2]=E[V(W_n)].
$$
As in Huang et al. \cite{H2013}, to prove Theorem \ref{thm1}, we first state the central limit theorem associated with the stochastic sequence of $\{W_n\}_{n\ge 1}$, which is a key step to establish Theorem \ref{thm1}.
\begin{lem}\label{lem1} Assume $\{X_n,n\ge0\}$ is a countable nonhomogeneous Markov chain taking values in $S=\{1,2,\ldots\}$ with initial distribution of (\ref{1}) and transition matrices of (\ref{2}). Suppose $f$ is a real function satisfying $|f(x)|\le M$ for all $x\in \rr$. Assume that $P$ is a periodic strongly ergodic stochastic matrix, and $R$ is a constant stochastic matrix each row of which is the left eigenvector $\pi=(\pi(1),\pi(2),\ldots)$ of $P$ satisfying $\pi P=\pi$ and $\sum_{i}\pi(i)=1$. Suppose that (\ref{4}) and (\ref{5}) are satisfied, and $\{W_n,n\ge0\}$ is defined by (\ref{2.2}).Then
\begin{equation}\label{2.3}
\frac{W_n}{\sqrt{n\theta}}\stackrel{D}{\Rightarrow}N(0,1),
\end{equation}
where $\stackrel{D}{\Rightarrow}$ stands for the convergence in distribution.
\end{lem}

As in Huang et al. \cite{H2013}, to establish Lemma \ref{lem1}, we need two important statements below such as Lemma 2.2 (see Brown \cite{Brown1971}) and Lemma 2.3 (see Yang \cite{Yang2009}).
\begin{lem}\label{lem2}
Assume that $(\Omega,\FF,\pp)$ is a probability space, and $\{\FF_n,n=1,2,\ldots\}$ is an increasing sequence of $\sigma$-algebras. Suppose that $\{M_n,\FF_n,n=1,2,\ldots\}$ is a martingale, denote its related martingale difference by $\xi_0=0$, $\xi_n=M_n-M_{n-1}$ $(n=1,2,\ldots)$. For $n=1,2,\ldots$, write
$$
V(M_n)=\sum_{j=1}^{n}E[\xi^2_j|\FF_{j-1}],
$$
$$
v(M_n)=E[V(M_n)],
$$
where $\FF_0$ is the trivial $\sigma$-algebra. Assume that the following holds:\\
(i)
\begin{equation}\label{2.4}
\frac{V(M_n)}{v(M_n)}\stackrel{P}{\Rightarrow}1,
\end{equation}
(ii) the Lindeberg condition holds, i.e., for any $\epsilon>0$
$$
\lim_{n\rightarrow\infty}\frac{\sum_{j=1}^{n}E[\xi^2_jI(|\xi_j|\ge \epsilon\sqrt{v(M_n)})]}{v(M_n)}=0,
$$
where $I(\cdot)$ denotes the indicator function.
Then we have
\begin{equation}\label{2.5}
\frac{V(M_n)}{\sqrt{v(M_n)}}\stackrel{D}{\Rightarrow}N(0,1),
\end{equation}
where $\stackrel{P}{\Rightarrow}$, and $\stackrel{D}{\Rightarrow}$ denote convergence in probability, and in distribution respectively.
\end{lem}

Write $\delta_i(j)=\delta_{ij}$, $(i,j\in S)$. Set
$$
L_n(i)=\sum_{k=0}^{n-1}\delta_i(X_k).
$$
\begin{lem}[]\label{lem3} Assume that $\{X_n,n\ge 0\}$ is a countable nonhomogeneous Markov chain taking values in $S=\{1,2,\ldots\}$ with initial distribution $(\ref{1})$, and transition matrices $(\ref{2})$. Suppose that $P$ is a periodic strongly ergodic stochastic matrix, and $R$ is matrix each row of which is the left eigenvector $\pi=(\pi(1),\pi(2),\cdots)$ of $P$ satisfying $\pi P=\pi$ and $\sum_{i}\pi(i)=1$. Assume (\ref{4}) holds. Then
\begin{equation}\label{2.6}
\lim_{n\rightarrow\infty}\frac1n L_n(i)=\pi(i) \mbox{       } a.e..
\end{equation}
\end{lem}
Now let's come to establish Lemma \ref{lem1}.

\begin{proof}[Proof of Lemma \ref{lem1}]
Applications of properties of the conditional expectation, and Markov chains yield
\begin{equation}\label{2.7}
\begin{aligned}
\frac{V(W_n)}{n}&=\frac1n\sum_{k=1}^{n}E[D_k^2|\FF_{k-1}]\\
&=\frac1n\sum_{k=1}^{n}\{E[f^2(X_k)|X_{k-1}]-[E[f(X_k)|X_{k-1}]]^2\}\\
&:=I_1(n)-I_2(n),
\end{aligned}
\end{equation}
where
\begin{equation}\label{2.8}
\begin{aligned}
I_1(n)&=\frac1n\sum_{k=1}^{n}E[f^2(X_k)|X_{k-1}]\\
&=\sum_{j\in S}\sum_{i\in S}f^2(j)\frac1n\sum_{k=1}^{n}p_k(i,j)\delta_i(X_{k-1}),
\end{aligned}
\end{equation}
and
\begin{equation}\label{2.9}
\begin{aligned}
I_2(n)&=\frac1n\sum_{k=1}^{n}[E[f(X_k)|X_{k-1}]]^2\\
&=\sum_{i\in S}\sum_{j,\ell\in S}f(j)f(\ell)\frac1n\sum_{k=1}^{n}p_k(i,j)p_k(i,\ell)\delta_i(X_{k-1}).
\end{aligned}
\end{equation}
We first use (\ref{4}), and Fubini's theorem to obtain
\begin{equation}\label{2.10}
\begin{aligned}&\lim_{n\rightarrow\infty}\sum_{i\in S}\frac1n\sum_{k=1}^{n}\sum_{j\in S}\delta_i(X_{k-1})|p_k(i,j)-p(i,j)|\\
&\le \lim_{n\rightarrow\infty}\frac1n \sum_{i\in S}\delta_i(X_{k-1})\sum_{k=1}^{n}\|P_k-P\|\\
&=\lim_{n\rightarrow\infty}\frac1n\sum_{k=1}^{n} \sum_{i\in S}\delta_i(X_{k-1})\|P_k-P\|\\
&\le \lim_{n\rightarrow\infty}\frac1n\sum_{k=1}^{n} \|P_k-P\|=0.
\end{aligned}
\end{equation}
Hence, it follows from (\ref{2.10}), and $\pi P=\pi$ that
\begin{equation}\label{2.11}
\begin{aligned}
\lim_{n\rightarrow\infty}I_1(n)&=\lim_{n\rightarrow\infty}\sum_{j\in S}\sum_{i\in S}f^2(j)\frac1n\sum_{k=1}^{n}p(i,j)\delta_i(X_{k-1})\\
&=\sum_{j\in S}\sum_{i\in S}f^2(j)p(i,j)\pi(i)\\
&=\sum_{j\in S}f^2(j)\pi(j),\mbox{     } a.e..
\end{aligned}
\end{equation}
We next claim that
\begin{equation}\label{2.12}
\lim_{n\rightarrow\infty}I_2(n)=\sum_{i\in S}\pi(i)\left[\sum_{j\in S}f(j)p(i,j)\right]^2,\mbox{    }  a.e..
\end{equation}
Indeed, we use (\ref{4}), and (\ref{2.9})  to have
$$
\begin{aligned}
&\left|I_2(n)-\sum_{i\in S}\sum_{j,\ell\in S}f(j)f(\ell)\frac1n\sum_{k=1}^{n}p(i,j)p(i,\ell)\delta_i(X_{k-1})\right|\\
&\le \left|\sum_{i\in S}\sum_{j,\ell\in S}f(j)f(\ell)\frac1n\delta_i(X_{k-1})\sum_{k=1}^{n}(p_k(i,j)-p(i,j))p_k(i,\ell)\right.\\
&\left.+\sum_{i\in S}\sum_{j,\ell\in S}f(j)f(\ell)\frac1n\delta_i(X_{k-1})\sum_{k=1}^{n}p(i,j)(p_k(i,\ell)-p(i,\ell))\right|\\
&\le M^2\left(\frac1n\sum_{k=1}^{n}\sum_{i\in S}\delta_i(X_{k-1})\|P_k-P\|+\frac1n\sum_{k=1}^{n}\sum_{i\in S}\delta_i(X_{k-1})\|P_k-P\|\right)\\
&\le 2M^2\frac1n\sum_{k=1}^{n}\|P_k-P\|\rightarrow0, \mbox{ as }n\rightarrow\infty.
\end{aligned}
$$
Thus, we use Lemma \ref{lem3} again to obtain
$$
\begin{aligned}
\lim_{n\rightarrow\infty}I_2(n)&=\sum_{i\in S}\sum_{j,\ell\in S}f(j)f(\ell)p(i,j)p(i,\ell)\lim_{n\rightarrow\infty}\frac1n\sum_{k=1}^{n}\delta_i(X_{k-1})\\
&=\sum_{i\in S}\sum_{j,\ell\in S}f(j)f(\ell)p(i,j)p(i,\ell)\pi(i)\\
&=\sum_{i\in S}\pi(i)[\sum_{j\in S}f(j)p(i,j)]^2  \mbox{  a.e.. }
\end{aligned}
$$
Therefore (\ref{2.12}) holds. Combining (\ref{2.11}), and (\ref{2.12}), results in
\begin{equation}\label{2.13}
\lim_{n\rightarrow\infty}\frac{V(W_n)}{n}=\sum_{i\in S}\pi(i)[f^2(i)-(\sum_{j\in S}f(j)p(i,j))^2] \mbox{   a.e. ,}
\end{equation}
which gives
\begin{equation}\label{2.14}
\lim_{n\rightarrow\infty}\frac{V(W_n)}{n}=\sum_{i\in S}\pi(i)[f^2(i)-(\sum_{j\in S}f(j)p(i,j))^2] \mbox{   in probability.}
\end{equation}
Since $\{V(W_n)/n,n\ge 1\}$ is uniformly bounded, $\{V(W_n)/n,n\ge 1\}$ is uniformly integrable. By applying the above two facts, and (\ref{5}) we have
\begin{equation}\label{2.16}
\lim_{n\rightarrow\infty}\frac{E[V(W_n)]}{n}=\sum_{i\in S}\pi(i)[f^2(i)-(\sum_{j\in S}f(j)p(i,j))^2] >0.
\end{equation}
Therefore we obtain
$$
\frac{V(W_n)}{v(W_n)}\stackrel{P}{\Rightarrow}1.
$$
Also note that
$$
\{D^2_n=[f(X_n)-E[f(X_n)|X_{n-1}]]^2\}
$$
is uiformly integrable. Thus
$$
\lim_{n\rightarrow\infty}\frac{\sum_{j=1}^{n}E[D^2_jI(|D_j|\ge\epsilon\sqrt{n} )]}{n}=0,
$$
which implies that the Lindeberg condition holds.  Application of Lemma \ref{lem2} yields (\ref{2.3}). This establishes Lemma \ref{lem1}.
\end{proof}
\begin{proof}[Proof of Theorem \ref{thm1}]
Note that
\begin{equation}\label{2.17}
S_n-E[S_n]=W_n+\sum_{k=1}^{n}[E[f(X_k)|X_{k-1}]-E[f(X_k)]].
\end{equation}
Write
$$
\pp(X_k=j)=P_k(j),j\in S.
$$
Let's evaluate the upper bound of $|E[f(X_k)|X_{k-1}]-E[f(X_k)]|$. In fact, we use the C-K formula of Markov chain to obtain
$$
\begin{aligned}
|E[f(X_k)|X_{k-1}]-E[f(X_k)]|&=\left|\sum_{j\in S}f(j)P_k(j|X_{k-1})-\sum_{j\in S}f(j)P_k(j)\right|\\
&\le \sup_{i}\left|\sum_{j\in S}f(j)\left[P_k(j|i)-\sum_{s}P_{k-1}(s)P_k(j|s)\right]\right|\\
&\le M\sup_{i}\sum_{j\in S}\left|P_k(j|i)-\sum_{s}P_{k-1}(s)P_k(j|s)\right|\\
&= M\sup_{i}\sum_{j\in S}\left|\sum_{s}P_{k-1}(s)P_k(j|i)-\sum_{s}P_{k-1}(s)P_k(j|s)\right|\\
&\le M\sup_{i}\sum_{s}P_{k-1}(s)\sup_{s}\sum_{j\in S}|P_k(j|i)-P_k(j|s)|\\
&= M\sup_{i,s}\sum_{j\in S}|P_k(j|i)-P_k(j|s)|\\
&=2M\delta(P_k),
\end{aligned}
$$
here
$$
\delta(P_k)=\sup_{i,s}\sum_{j\in S}[P_k(j|i)-P_k(j|s)]^{+}=\frac12\sup_{i,s}\sum_{j\in S}|P_k(j|i)-P_k(j|s)|.
$$
Application of (\ref{6}) yields
\begin{equation}\label{2.18}
\lim_{n\rightarrow\infty}\frac{\sum_{k=1}^{n}[E[f(X_k)|X_{k-1}]-E[f(X_k)]]}{\sqrt{n}}=0.
\end{equation}
Combining (\ref{6}), (\ref{2.3}), (\ref{2.17}), and (\ref{2.18}), results in (\ref{7}). This proves Theorem \ref{thm1}.
\end{proof}
\section{Proof of Theorem \ref{thm2}}
As in Xu et al. \cite{Xu2020}, we use Theorem 1.1 in Gao \cite{Gao1996}, and exponential equivalence methods to prove Theorem \ref{thm2}. By (\ref{2.13}),
$$
\begin{aligned}
&\liminf_{n\rightarrow\infty,\frac{n}{m}\rightarrow0}\sup_{j\ge 0}\frac{1}{m}t^2E\left(\sum_{i=1}^{m}D_{n+j+i,f}^2|\FF_j\right)-{t^2}\theta(f)\\
 &=\liminf_{n\rightarrow\infty,\frac{n}{m}\rightarrow0}\sup_{j\ge 0}E\left(\frac{1}{m}t^2E\left(\sum_{i=1}^{m}D_{n+j+i,f}^2|\FF_{n+j+i-1}\right)|\FF_j\right)-t^2\theta(f)=0\mbox {  a.e.,}
\end{aligned}
$$
 which by the uniform boundedness of $\frac{1}{m}\sum_{k=1}^{m}E\left(\sum_{i=1}^{m}D_{n+j+i,f}^2|\FF_j\right)$ implies
 \begin{equation}\label{2.20}
\liminf_{n\rightarrow\infty,\frac{n}{m}\rightarrow0}\sup_{j\ge 0}\left\|\frac{1}{m}t^2E\left(\sum_{i=1}^{m}D_{n+j+i,f}^2|\FF_j\right)-{t^2}\theta(f)\right\|_{L^{\infty}(\pp)}=0.
\end{equation}
It is obvious that there exists $ \delta>0$ such that
 \begin{align*}
&\sup_{m\ge0}\left\|E(\exp(\delta D_{m+1,f})|\FF_{m})\right\|_{L^{\infty}(\pp)}<\infty,
 \end{align*},
 which together with (\ref{2.20}), Theorem 1.1 in Gao \cite{Gao1996}, G\"artner-Ellis theorem implies that $W_n/a(n)$ satisfies the moderate deviation theorem with rate function
$I(x)=\frac{x^2}{2\theta}$.
It follows from (\ref{8}) and (\ref{2.16}) that $\forall \epsilon>0$,
\begin{align*}
&\lim_{n\rightarrow\infty}\frac{n}{a^2(n)}\log \pp\left(\left|\frac{S_n-E[S_n]}{a(n)}-\frac{W_n}{a(n)}\right|>\epsilon\right)\\
&=\lim_{n\rightarrow\infty}\frac{n}{a^2(n)}\log \pp\left(\left|\frac{\sum_{k=1}^{n}[E[f(X_k)|X_{k-1}]-E[f(X_k)]]}{a(n)}\right|>\epsilon\right)\\
&=0.
 \end{align*}
Thus, by the exponential equivalent method (see Theorem 4.2.13 of Dembo and Zeitouni \cite{Dembo1998}, Gao \cite{Gao2013}),  we see that $\{\frac{S_n-E[S_n]}{a(n)}\}$ satisfies the same moderate deviation theorem as $\{\frac{W_n}{a(n)}\}$ with rate function $I(x)=\frac{x^2}{2\theta}$. This completes the proof .

{\bf Acknowledgements.} The authors are very grateful to the referees for carefully reading the original
manuscript and giving valuable suggestions.


\begin{thebibliography}{123}
\bibitem{H2013}Huang H L, Yang W G, Shi Z Y. The central limit Theorem for nonhomogeneous Markov chains[J].  Chinese Journal of Applied Probability and Satistics, 2013, 29(4): 337-347.
\bibitem{Gao1992}Gao F Q. Moderately large deviations for uniformly ergodic markov processes, research announcements[J]. Advances in Mathematics(China), 1992, 21(3): 364-365.
\bibitem{Gao1996}Gao F Q.  { Moderate deviations for martingales and mixing random processes}[J]. {Stochastic Process. Appl.},1996, {61}, 263-275.

\bibitem{deAcosta1997}de Acosta A. Moderate deviations for empirical measures of markov chain: lower bounds[J]. Ann. Prob., 1997, 25(1): 259-284.
 \bibitem{deAcosta1998}de Acosta A, Chen X. Moderate deviations for empirical measures of markov chain: upper bounds[J].  J. Theoret. Probab., 1998, 11(4): 1075-1110.
 \bibitem{Bowerman1977}Bowerman B, David H T and Isaacson D. The convergence of Ces\`{a}ro averages for certain nonstationary Markov chains[J].  Stochastic Process. Appl., 1977, {5}(1): 221-230.
 \bibitem{Hu1983}Hu D H. Countable Markov process theory[M]. Wuhan: Wuhan University Publishing house(in Chinese), 1983.
\bibitem{Yang2002}Yang W G. Convergence in the Ces\`{a}ro sense ans strong law of large numbers for nonhomogeneous Markov chains[J].  Linear Algebra Appl., 2002, {354}(1): 275-286.
\bibitem{Brown1971}Brown B M.  Martingale central limit theorems[J]. Ann. Math. Statist., 1971, {42}(1): 59-66.
\bibitem{Yang2009}Yang W G. Strong law of large numbers for nonhomogeneous Markov chains[J]. Linear Algebra Appl., 2009, {430}(11-12): 3008-3018.
 \bibitem{Xu2020}Xu M Z, Ding Y Z, Zhou Y Z. Moderate deviation for nonhomogeneous Markov chains[J].Statistics and Probability Letters, 2020, 157, 10862.
\bibitem{Dembo1998}Dembo A, Zeitouni O.  Large Deviations Techniques and Applications[M]. New York: Springer, 1998.
\bibitem{Gao2013}Gao F Q. Moderate deviations for a nonparametric estimator for sample coverage[J]. Ann. Prob., 2013 {41}(2): 641-669.
\bibitem{Zhang2016}Zhang H Z, Hao R L, Ye Z X, Yang W G. Some strong limit properties for countable nonhomogeneous Markov chains (in Chinese)[J]. Chinese Journal of Applied Probability and Satistics, 2016, {32}(1): 62-68.

\end{thebibliography}
\end{document}